\documentclass[12]{amsart}
\usepackage[latin1]{inputenc}
\usepackage[T1]{fontenc}
\usepackage{lmodern}
\usepackage[english]{babel}
\usepackage{amsmath}
\usepackage{amsthm}
\usepackage{amsfonts, amssymb, amscd}
\usepackage{layout}
\usepackage{enumerate}
\usepackage{ifpdf}
\ifpdf
\usepackage[pdftex]{graphicx}
\DeclareGraphicsExtensions{.pdf} \else
\usepackage[dvips]{graphicx}
\DeclareGraphicsExtensions{.eps} \fi

\theoremstyle{plain}
\newtheorem*{teo*}{Theorem}

\newtheorem{teo}{Theorem}[section]

\newtheorem{lem}[teo]{Lemma}

\newtheorem{remark}[teo]{Remark}

\newcommand{\N}{{\rm I\!N}}
\newcommand{\R}{{\rm I\!R}}

\newcommand{\virg}[1]{``#1''}

\numberwithin{equation}{section}

\title[Simultaneous Resolution of singularities in the Nash Category]{Simultaneous Resolution of singularities in the Nash Category: Finiteness and effectiveness}

\author{ Demdah Kartoue Mady}
\thanks{The author is partially supported by CONFOFOR/TCHAD fund}

\date{}
\begin{document}
\setlength{\textheight}{44pc}
\setlength{\textwidth}{28pc}
\begin{abstract}
In this paper we present new proofs using real spectra of the finiteness theorem on Nash trivial simultaneous resolution and the finiteness theorem on Blow-Nash triviality for  isolated real algebraic singularities. That is, we prove that a family of Nash sets in a Nash manifold indexed by a semialgebraic set always admits a  Nash trivial simultaneous resolution after a partition of the parameter space into finitely many semialgebraic pieces and in the case of isolated singularities it admits a finite Blow-Nash trivialization. We also complement the finiteness results with recursive bounds. 

\end{abstract}

\maketitle

%\tableofcontents
\section*{Introduction}
 In [CS], M.Coste and M. Shiota proved Nash triviality in family, which can be seen as the Nash version of the Ehresmann's Lemma which gives  a trivialization of a proper submersion. The additional difficulty in obtaining Nash triviality is that integration of vector fields cannot be used. 
 Several authors (Fukui, Koike, Shiota...) have obtained similar results which are useful for classification of singularities, for instance in the study of Blow-analytic equivalence. In this paper we give  new proofs of such  results using real spectra.

 Let $N$ be an {\it affine} Nash manifold (i.e. a Nash submanifold of affine space $\mathbb{R}^n$) and $J$ be a semialgebraic  subset of $\mathbb{R}^s$. The notion of semialgebraic family of maps is more exactly defined by $F: N\times J\rightarrow \mathbb{R}^k\times J$ and its domaine and its target need not to be trivial semialegbraic families (see \cite{BCR} p.149). But here a mapping $F: N\times J\rightarrow \mathbb{R}^k$ will be said a semialgebraic family of Nash maps if $F$ is semialgebraic and for any $t\in J$ the map $f_t: N \rightarrow \mathbb{R}^k$ defined by $f_t(x)=F(x,t)$ is  Nash. Here only semialgebraic families of Nash functions with target $\mathbb{R}^k$ are considered, so there is no need to write the target of $F$ the family $\mathbb{R}^k\times J,$ but just $\mathbb{R}^k.$ One can note also that the domaine of $F$ is a trivial semialgebraic family. 

For $Q\subset J$, we denote $F|_{N\times Q}$ by $F_Q$. Suppose $F_Q$ is Nash. Then we say that the family of zero-sets $F_Q^{-1}(0)$ admits a {\it Nash trivial simultaneous resolution} along $Q$ if there exists a desingularization of the {\it Nash set}   defined by $F_Q = 0$ which is Nash trivial. 
 We say also that the family of zero-sets $F_Q^{-1}(0)$ admits a  {\it Blow-Nash trivialization}  along $Q$  if there exist a desingularization of the {\it Nash set}  defined
by $F_Q = 0$ and a Nash triviality upstairs  which induces a semialgebraic
trivialization of the zero-set $F_Q^{-1}(0)$ (for precise definitions, see Section 1 ). Here are the results we prove in this paper.
\begin{teo*}[I]
Let $F: N\times J\rightarrow \mathbb{R}^k$ be semialgebraic family of Nash maps from $N$ to $\mathbb{R}^k.$ Then there exists a finite partition $J=J_1\cup ...\cup J_p$ which satisfies the following conditions:
\begin{enumerate}
\item    Each $J_i$ is a Nash manifold which is Nash diffeomorphic to an open  simplex in some Euclidean space, and $F_{J_i}$ is  Nash.
\item  For each $i$, $F_{J_i}^{-1}(0)$ admits a Nash trivial simultaneous resolution along $J_i.$
 \end{enumerate}
\end{teo*}
When for all $t\in J,$ $f_t^{-1}(0)$ has only isolated singularities, the following result holds.
\begin{teo*}[II]
Let $F: N\times J\rightarrow \mathbb{R}^k$ be semialgebraic family of Nash maps from $N$ to $\mathbb{R}^k$  such that for all $t\in J,$ $f_t^{-1}(0)$ has only isolated singularities. Then there exists a finite partition $J=J_1\cup ...\cup J_p$ which satisfies the following conditions:
\begin{enumerate}
\item    Each $J_i$ is Nash manifold which is Nash diffeomorphic to an open  simplex in some Euclidean space, and $F_{J_i}$ is Nash.
\item  For each $i$, $F_{J_i}^{-1}(0)$ admits a Blow-Nash trivialization along $J_i.$
 \end{enumerate}
\end{teo*}
In the setting of Nash germs, these theorems were first proved by Fukui, Koike and Shiota (cf.[FKS], [K1]). 
In this paper, we obtain the same results, but in the global case with $N$ an affine Nash manifold not necessarily compact using the techniques of real spectra.
First, we verify that the Desingularization Theorem holds for the category of Nash sets over an arbitrary real closed field, equipped with a sheaf of Nash functions. The precise definition of this category needs the use of the real spectrum and of finite sheaves of ideals of Nash functions over this real spectrum (see Section 2). That is, we check that the above category verifies the  criteria given by Bierstone and Milman (cf. [BM1] p.234) to apply the Desingularization Theorem. We apply the Desingularization Theorem to the fiber of the family at a point of the real spectrum of the parameter space. This fiber is defined over a real closed field which is not $\mathbb{R}$ a priori, which explains the need to have a desingularisation valid over any real closed field. The finiteness follows from the compactness property of the real spectrum of the parameter space.  

The paper is organized as follows. In the first section we recall  useful results and give the definition of  blow-Nash triviality.  In the second section, we recall the notion of real spectrum and some related results which will be used. In section three, we  check that the  desingularization theorem holds true over an arbitrary real closed field. In particular, we prove that glueing local Blowing-up of an affine Nash manifold gives an affine Nash manifold (we need always to work with affine Nash manifolds). As a consequence of the Section 3, we prove in section four the finiteness of Nash trivial simultaneous resolution (\textbf{Theorem (I)}). Finally, in the last section we prove the \textbf{Theorem~(II)}.

{\it Acknowledgment}  - I would like to thank Michel Coste for his advice during the preparation of this paper. I also would like to thank anonymous referee for reading carefully the previous version and for his remarkable suggestions.

\section{Preliminaries}
In this paper, a {\it Nash manifold} is a semialgebraic  subset of $\R^n$ (or of $R^n$, where $R$ is a real closed field) which is also a $C^{\infty}$ submanifold. For more details see \cite{S} or \cite{BCR} for the use of an arbitrary real closed field. We 
shall work in this paper only with {\it affine} Nash manifolds to avoid some pathological situations. We define  a {\it Nash subset} (or a  {\it Nash variety})  of a Nash manifold $M$ as  a subset of $M$ which is the zero set of a Nash function over $M$.

Let us now introduce the notion of normal crossing that we will use in this paper. Let  $M$ be a manifold of dimension $m$. Let $N_1$,...,$N_r$ be  Nash proper submanifolds  in $M$. Let $P\in \displaystyle\cap_{i=1}^rN_i.$
 The submanifolds $N_1$,...,$N_r$   are said to have {\it simultaneous normal crossings at a point $P$}  when there exist $(x_1,...,x_m)$ a local coordinates
  system for $M$ at $P,$  $J_1,...,J_r$ disjoint subsets of $\{1,...,m\}$ such that
$N_i=\{x_j=0\,\, \mbox{for}\,\, j\in J_i\}$ for $i=1,...,r.$

Now consider $N_1$,..., $N_u$ a finite family of Nash proper submanifolds of $M.$ The manifolds $N_1$,..., $N_u$ are said to have {\it only simultaneous normal crossings} if for all $P\in M$, the submanifolds $N_i$ which contain $P$ have simultaneous normal crossings at $P.$

Let $M$ and $\mathcal{M}$ be Nash manifolds, and let $V$ be a Nash subset of $M$. A map $\Pi: \mathcal{M}\rightarrow M$ is said to be a {\it Nash resolution}   of  $V$ in $M$ if there exists a finite sequence of blowings-up  $\sigma_{j+1}: \mathcal{M}_{j+1} \rightarrow \mathcal{M}_j$ with smooth centers $C_j$ (where $\sigma_j$, $\mathcal{M}_j$ and $C_j$ are of Nash class)  such that:
\begin{itemize}
\item $\Pi$ is the composite of $\sigma_j$'s;
\item the critical set of $\Pi$ is a union of smooth Nash divisors $\mathcal{D}_1$,..., $\mathcal{D}_r$;
\item the strict transform of $V$ in $M$ by $\Pi$ (denoted by $\mathcal{V}$)  is a Nash submanifold of $\mathcal{M}$;
\item $\mathcal{V}$, $\mathcal{D}_1$,..., $\mathcal{D}_r$ simultaneously have only normal crossings.
\end{itemize}
It is folklore that the desingularization theorem (cf.\cite{H1}, \cite{H2}, \cite{BM1} and \cite{BM2}) apply in the Nash category to give:
\begin{teo} (Desingularization )\label{desing}\\
For a Nash subset $V$ of a Nash manifold $M$, there exists a Nash resolution $\Pi: \mathcal{M}\rightarrow M$ of $V$ in $M.$
\end{teo}
We shall make precise the Nash category in the following section, working over an arbitrary real closed field.

Now let $M, J$  be Nash manifolds and $V$ a Nash subset of $M$.  Let $\Pi: \mathcal{M}\rightarrow M$ be a Nash resolution of $V$ and $q: M \rightarrow J$ be an onto Nash submersion. For any $t\in J$, let $\mathcal{M}_t= (q\circ \Pi)^{-1}(t)$, $M_t= q^{-1}(t)$ and $V_t= V\cap M_t.$ The map $\Pi$ is said to be a {\it Nash simultaneous resolution}  of $V$  in $M$ over $J$ if  

  \begin{itemize}
   \item $ q\circ \Pi|_{\mathcal{V}}: \mathcal{V}\rightarrow J,$ 
   \item  $q\circ \Pi|_{\mathcal{D}_{j_1}\cap...\cap \mathcal{D}_{j_s}}: \mathcal{D}_{j_1}\cap...\cap \mathcal{D}_{j_s} \rightarrow J,$
    \item $q\circ \Pi|_{\mathcal{V}\cap \mathcal{D}_{j_1}\cap...\cap \mathcal{D}_{j_s}}:\mathcal{V}\cap \mathcal{D}_{j_1}\cap...\cap  \mathcal{D}_{j_s}\rightarrow J $   $(1\leq j_1<...<j_s\leq r)$
  \end{itemize}
are onto submersions;

Now let $\Pi: \mathcal{M}\rightarrow M$ be a Nash simultaneous resolution of $V$ 
over $J$, and let $t_0\in J$. 
We say that $\Pi$ gives a {\it Nash trivial simultaneous resolution}  of $V$ in $M$ over $J$  if there is a Nash diffeomorphism $\Phi: \mathcal{M} \rightarrow \mathcal{M}_{t_0}\times J$ such that 
\begin{enumerate}
\item $(q\circ \Pi)\circ \Phi^{-1}:\mathcal{M}_{t_0}\times J \rightarrow J$ is the natural projection,
\item $\Phi(\mathcal{V})=\mathcal{V}_{t_0}\times J$, 
$\Phi(\mathcal{D}_i)=\mathcal{D}_{i,t_0}\times J,$
\item $\Phi^{-1}(x,t_0)=x,$ for every $x\in \mathcal{M}_{t_0}.$
\end{enumerate}
 When the Nash trivialization $\Phi$ induces a semialgebraic trivialization of $V$ in $M$ over $J$, we say that  $(M, V)$  admits a {\it $\Pi$-Blow-Nash trivialization along}  $J$. This is equivalent to saying that there is a semialgebraic homeomorphism $\phi$ such that the following diagram commutes:
$$\begin{CD}
\mathcal{M} @>{\Pi_i}>> M@>q>>J\\
@V{\Phi}VV @V{\phi}VV @|\\
\mathcal{M}_{t_0}\times J@>{{\Pi}_{|{\mathcal{M}_{t_0}}}\times id_J}>> M_{t_0}\times J@>p>> J
\end{CD}$$
where $p(x,t)=t$ for $(x,t)\in M_{t_0}\times J.$

\section{Real spectrum, Families of Nash functions}
In order to work over any real closed field, real spectrum is necessary. So, in this section we recall certain facts about Nash sheaves and real spectrum. 

Let $M$ be a Nash manifold defined over a real closed field $R$. In general, one cannot define the sheaf of Nash functions over $M$ by the order topology, because this topology  can be totally disconnected. Consider for example the locally Nash function $f$ over $\mathbb{R}_{alg}$ defined by $f=0$ over $\mathbb{R}_{alg}\cap (-\infty, \pi)$ and $f=1$ over $\mathbb{R}_{alg}\cap (\pi, +\infty)$: it  is not a Nash function over $\mathbb{R}_{alg}$ since it is not semialgebraic over $\mathbb{R}_{alg}.$ For $R=\mathbb{R}$ the sheaf  $\mathcal{N}$  of rings of Nash functions germs on $M$ is well defined, but in the non-compact case another problem may appear. The  sheaf $\mathcal{N}$ is coherent, but in \cite{LT}, the authors observe that there is a coherent sheaf of ideals of germs of Nash functions which is not generated by its global sections. For instance, the  sheaf $\mathcal{I}_{\mathbb{Z}}$ of germs of Nash functions vanishing on  $\mathbb{Z}\subset \mathbb{R}$ is not generated by  global Nash functions.

Let us recall now the notion of real spectrum which will be used in this paper. Let us consider 
$M\subset R^v$ a semialgebraic subset. Define 
$\widetilde{M}$ to be the set of  ultrafilters of the Boolean algebra of semialgebraic subsets of  $M.$
To any $x\in M,$ we associate $\alpha_x$ the principal ultrafilter of semialgebraic subsets of $M$ containing $x.$ If $S$ is a semialgebraic subset of  $M,$ $\widetilde{S}$ may be identified with the subset of $\widetilde{M}$ consisting of those ultrafilters $\alpha\in \widetilde{M}$ such that $S\in \alpha.$  The real spectrum has the following compactness property: 
for any covering $\widetilde{M}=\displaystyle\bigcup_{i\in I}\widetilde{S_i}$ where $S_i$ are semialgebraic subset of $M$, there exists a finite subset $J$ of $I$ such that $\widetilde{M}=\displaystyle\bigcup_{i\in J}\widetilde{S_i}.$ We define a topology on $\widetilde{M}$  by taking as basis of open subsets all $\widetilde{U}$ where $U$ is a semialgebraic open subset of $M.$ Denote it $\mathcal{T}^{sa}.$

Now, let $M$ be a Nash submanifold of $R^n.$ There exists a  unique  sheaf of rings $\mathcal{N}$ on $\widetilde{M}$ such that 
for every semialgebraic open subset $U$ of $M$, $\mathcal{N}(\widetilde{U})=\mathcal{N}(U)$ the ring of Nash functions on $U.$ For $\alpha \in \widetilde{M}$, the stalk of $\mathcal{N}$ at $\alpha$ is defined by 
$\mathcal{N}_{\alpha}=\displaystyle \underset{\alpha \in \widetilde{U}}{\underset{\longrightarrow}{\lim}} \mathcal{N}(U).$
The proof of the existence and the uniqueness of the sheaf  $\mathcal{N}$ can be found in [\cite{BCR}, Proposition 8.8.2] and the fact that, for any  $\alpha \in \widetilde{M}$, the stalk $\mathcal{N}_{\alpha}$ is a henselian local ring. Its residue field is denoted by  $k(\alpha)$ which  is a real closed extension of $R$. If $\alpha=x\in M,$ of course $\mathcal{N}_{\alpha}=\mathcal{N}_x$ the ring of germs of Nash functions at $x$ and $k(x)=R$ (cf. \cite{BCR}, Proposition 8.8.3). Recall that the dimension of $\alpha$ is the smallest dimension of $U$ for $U\in \alpha.$  Consider the completion $\widehat{\mathcal{N}}_{\alpha}$   of the ring  $\mathcal{N}_{\alpha},$ then  by \cite{BCR} p.162, there is an injective mapping
$ T_{\alpha}:\mathcal{N}_{\alpha}\rightarrow k(\alpha)[[X_1,...,X_r]]$ which maps each element of $\mathcal{N}_{\alpha}$ to its Taylor expansion, such that $T_{\alpha}$ induces an isomorphism from the completion 
$\widehat{\mathcal{N}}_{\alpha}$ to  $k(\alpha)[[X_1,...,X_r]]$ where $r=m-\mbox{dim}(\alpha).$ 
  
Let us denote by $\textbf{Nash}_R$ the Nash category over a real closed field $R.$
On the Nash manifold $\widetilde{M}$ (a smooth space in the sense that 
for all $\alpha \in \widetilde{M},$  $\mathcal{N}_{\alpha}$ is a regular local ring as consequence of above paragraph) let us consider a locally finite sheaf of ideals 
$\mathcal{I}\subset \mathcal{N},$ that is,  a sheaf of ideals locally (for the topology $\mathcal{T}^{sa}$) generated by finitely many Nash sections.  Then there exists an ideal $I$ of the ring of Nash function $\mathcal{N}(M)$ such that $\mathcal{I}=I.\mathcal{N}_{\widetilde{M}}$  and  $(\mathcal{N}/\mathcal{I})(M)=\mathcal{N}(M)/I$ (see \cite{CRS}). Thus one can note that the problems indicated in the beginning of this section are settled with the use of the real spectrum. So, we have that  $Z(\mathcal{I})=\widetilde{ Z(I)}$ is a Nash subset of $\widetilde{M}.$  Then in $\textbf{Nash}_R$, $(\widetilde{ Z(I)},\mathcal{N}/{\mathcal{I}})$ is a closed subspace of $\widetilde{M}.$  The family of Nash  sets on $\widetilde{M}$ defines a Zariski topology which is  noetherian.

The {\it objects} in the category  $\textbf{Nash}_R$ are locally ringed spaces $(\widetilde{X}, \mathcal{N}_{\widetilde{X}})$  which are  closed subspaces $(Z(\mathcal{I}), \mathcal{N}_{\widetilde{M}}/\mathcal{I})$ of a Nash manifold  $(\widetilde{M},\mathcal{N}_{\widetilde{M}}).$  Given  $(\widetilde{X}, \mathcal{N}_{\widetilde{X}})$ and $(\widetilde{Y}, \mathcal{N}_{\widetilde{Y}})$ two objects of $\textbf{Nash}_R,$ a {\it morphism} in $\textbf{Nash}_R$ from $(\widetilde{X}, \mathcal{N}_{\widetilde{X}})$ to $(\widetilde{Y}, \mathcal{N}_{\widetilde{Y}})$  is a pair $(f,f^{\#})$ of a continuous  semialgebraic map $f: \widetilde{X}\rightarrow \widetilde{Y}$ such that for any open semialgebraic set $U\subset \widetilde{Y}$and any Nash function $g\in \mathcal{N}_{\widetilde{Y}}(U),$ $f\circ g \in \mathcal{N}_{\widetilde{X}}(f^{-1}(U)),$  and a morphism $f^{\#}:\mathcal{N}_{\widetilde{Y}}\rightarrow f_{*}\mathcal{N}_{\widetilde{X}}$ of sheaves of rings on $\widetilde{Y}$ (where $f_{*}\mathcal{N}_{\widetilde{X}}$ is defined by $f_{*}\mathcal{N}_{\widetilde{X}}(U)=\mathcal{N}_{\widetilde{X}}(f^{-1}(U))$ for any open semialgebraic set $U\subset \widetilde{Y}$) such that, for each $x\in \widetilde{X},$ the induced map $f_{x}^{\#}:\mathcal{N}_{\widetilde{Y},f(x)}\rightarrow f_{*}\mathcal{N}_{\widetilde{X},x}$ is a local homomorphism of local rings.
 
 Now, let us fix some notations about semialgebraic family which will be used. Let $S$ be a semialgebraic subset of $R^s$ and $U\subset R^n\times S$ be a semialgebraic set. Then one can write 
$$U=\{(x,t)\in R^n\times S: \Phi(x,t) \}$$ with $\Phi(x,t)$ a first order formula in the theory of real closed fields with parameter in $R.$ Denote $t(\alpha)=(t_1(\alpha),...,t_s(\alpha))\in k(\alpha)^s$ where $t_i(\alpha)$, with $i=1,...,s$, are the images in $k(\alpha)$ of the coordinate functions $t$ in $R^s.$ Then for $\alpha\in \widetilde{S}$ set 
$$U_{\alpha}=\{x\in k(\alpha)^n: \Phi(x,t(\alpha))\}$$ which is a semialgebraic subset of $k(\alpha)^n$ (the fiber of the semialgebraic family $U$ at $\alpha$).

Reciprocally, if $\Omega\subset k(\alpha)^n$ is a semialgebraic set, there is a semialgebraic set $S$ with $\alpha\in \widetilde{S}$ and a semialgebraic set $U\subset R^n\times S$ such that $U_{\alpha}=\Omega.$ We have also that $U_{\alpha}=V_{\alpha}$ if and only if there exists a semialgebraic set $S$ with $\alpha\in \widetilde{S}$ such that $U_S=V_S.$

Hence, any semialgebraic subset of $k(\alpha)^n$ can be represented as a fiber at $\alpha$ of a semialgebraic family of semialgebraic subsets of $R^n.$ Two families of semialgebraic sets  have the same fiber at $\alpha$ if and only if they coincide over a semialgebraic set $S$ with $\alpha\in \widetilde{S}.$ After that, we make an abuse of notation writing $U_{\alpha}$ the semialgebraic subsets defined in $k(\alpha)^n.$

There is also a notion of semialgebraic family of mappings with similar properties (see \cite{BCR}, Proposition 7.4.4 and Remark 7.4.5).

The philosophy behind the use of fibers at points of the real spectrum is that any property which can be written in the first order language of real closed fields is true at $\alpha$ if and only if it is true at all $t\in S$ for some semialgebraic set such that $\alpha \in \widetilde{S}.$ This applies for instance for the property of being the graph of a semialgebraic function, and so one can define the fiber $f_{\alpha}$ for a semialgebraic  family of semialgebraic functions $(f_{t})_{t\in S},$ for any $\alpha\in\widetilde{S}.$

The following theorems will be used to prove our main results.

We consider, in the following Theorem,  only family of functions with target a trivial family but one can find a more general  statement
in \cite{BCR}, Proposition 8.10.1. 
\begin{teo} \label{family1}
Let $S\subset R^s$ be a semialgebraic set and $\alpha\in \widetilde{S}.$ Let $U$ be an open semialgebraic subset of $R^m\times S,$ and $F:U\rightarrow R,$ a semialgebraic family of functions parametrized by $S.$ If, for every $t\in S$ the fiber $f_t:U_t\rightarrow R$ is a Nash function, then for any $\alpha\in \widetilde{S},$ the fiber $f_{\alpha}:U_{\alpha}\rightarrow k(\alpha)$ is also a Nash function.
\end{teo}

By \cite{CS}, the ultrafilter $\alpha\in \widetilde{R}^s$ are generated by Nash submanifolds of $R^s$ (not closed), Nash diffeomorphic to open simplices.  Then we have the following theorems.
\begin{teo}(\cite{BCR}, Proposition 8.10.3)\label{family}\\
Let $\alpha \in \widetilde{R^s}$, and let $U_{\alpha}$ be an open semialgebraic subset of $k(\alpha)^m$, and $f_{\alpha}: U_{\alpha} \rightarrow k(\alpha)$, a Nash function. Then there exist a Nash submanifold $Q\subset R^s$, with $\alpha \in \widetilde{Q}$, an  open semialgebraic subset $U_Q$ of $R^m\times Q$ and  a semialgebraic family of functions $f_Q:U_Q\rightarrow R$ parametrized by $Q$, such that $f_Q$ is a Nash mapping whose fiber at $\alpha$ is $f_{\alpha}.$
\end{teo}

\begin{teo} (\cite{CS}, Proposition 2.2)\label{submersion}\\
Let $B\subset R^s$ be a semialgebraic set, and let $X$ be a semialgebraic subset of $R^n\times B.$ Let $\alpha \in \widetilde{B}$ be such that $X_{\alpha}$ is a Nash submanifold of $k(\alpha)^n.$ Then there is a Nash submanifold $Q$ of $R^s$, $Q\subset B$, such that $\alpha \in \widetilde{Q}$, $X_Q$ is a Nash submanifold of $R^n\times Q$, and the projection $p: X_Q\rightarrow Q$ is a submersion.
\end{teo}
We prove the following technical lemma.
\begin{lem}\label{croisement normaux}

Let $\alpha \in \widetilde{R^s}$, and let $N_{1\alpha}$ and $N_{2\alpha}$ be closed Nash submanifolds of  an affine  Nash manifold $M_{\alpha}.$  Then
\begin{enumerate}[(i)]
	\item If there is $S\subset R^s$ semialgebraic subset with $\alpha\in \widetilde{S}$ such that for all $t\in S,$  $N_{1t}$ and $N_{2t}$ are closed Nash submanifolds of $M_t$ with only normal crossings, then $N_{1\alpha}$ and $N_{2\alpha}$ have only normal crossings.
	\item If $N_{1\alpha}$ and $N_{2\alpha}$ have only normal crossings, then there exists a Nash sub\-manifold $Q\subset R^s$, with $\alpha \in \widetilde{Q}$ such that $N_{1Q}$ and $N_{2Q}$   have only normal crossings and the projections $N_{1Q}\rightarrow Q,$  $N_{2Q}\rightarrow Q$  and $N_{1Q}\cap N_{2Q}\rightarrow Q$ are Nash submersions.
\end{enumerate} 
\end{lem}
\begin{proof}

\begin{enumerate}[(i)]
\item Assume that there is a semialgebraic subset $S\subset R^s$ with $\alpha\in \widetilde{S}$ such that for all $t\in S,$  $N_{1t}$ and $N_{2t}$ are closed Nash submanifolds of $M_t$ with only normal crossings. For $t\in S,$ note that $N_{1t}$ and $N_{2t}$  have only normal crossings is equivalent to saying that for any $x\in N_{1t}\cap N_{2t},$ the codimension of $T_xN_{1t}\cap T_xN_{2t}$ is the sum of the codimensions of $T_xN_{1t}$ and $T_xN_{2t}$which can be expressed by a first order formula of the theory of real closed fields. In fact, being a tangent vector at $x$  and the linear independence of vectors can be expressed into a first order formula of the theory of real closed fields. It follows that $N_{1\alpha}$ and $N_{2\alpha}$ have only normal crossings.
\item One first application of Theorem \ref{submersion} gives $Q$ such that  $N_{1Q}$ and $N_{2Q}$ are closed Nash submanifolds of $M_Q$ and the projections $N_{1Q}\rightarrow $Q and $N_{2Q}\rightarrow Q$ are submersions. Then by
the first order description of normal crossings given above, shrinking $Q$ we can make $N_{1Q}$ and
$N_{2Q}$ normal crossings. But then $N_{1Q}\cap N_{2Q}$  is a closed submanifold, and again by Theorem \ref{submersion} we can
shrink Q to make $N_{1Q} \cap N_{2Q} \rightarrow Q$ a submersion. This completes the proof.
\end{enumerate}
\end{proof}
\begin{remark}
{\rm One can easily verify Lemma \ref{croisement normaux} in general case  with an arbitrary finite number of Nash submanifolds.}
\end{remark}

Now, let us recall the Nash Isotopy  Theorem.
Let $M$ be a Nash manifold and  $\varpi: M\rightarrow \mathbb{R}^s$ be a Nash submersion. For $Z\subset M,$ $Q\subset \mathbb{R}^s$ and $t\in \mathbb{R}^s$ let us denote by $Z_Q=Z\cap \varpi^{-1}(Q)$ and  
$Z_t=Z\cap \varpi^{-1}(t).$ The following result is proved in \cite{K2}:
\begin{teo}(\cite{K2}, Theorem I)\label{isotopy}\\
Let $N_1,..., N_r$  be Nash submanifolds of $M$ which have only normal crossings.
Let $\varpi: M\rightarrow \mathbb{R}^s,$ $s>0$, be an onto Nash submersion such that 
for every $0\leq i_1<...<i_q\leq r$, 
$\varpi|_{N_{i_1}\cap...\cap N_{i_q}}:N_{i_1}\cap...\cap N_{i_q}\rightarrow \mathbb{R}^s$ is an  onto submersion (unless $N_{i_1}\cap...\cap N_{i_q}=\emptyset $).
Then there exists a finite partition of $\mathbb{R}^s$ into Nash simplices $Q_j,$ $j=1,..., u$ and for each $j$ a point $t_j\in Q_j$ and a Nash diffeomorphism
$$\varphi_j= (\varphi_j', \varpi_j): M_{Q_j}\rightarrow M_{t_j}\times Q_j $$
such that $\varpi|_{M_{Q_j}}=\varpi_j|_{M_{Q_j}}$, $\varphi_j'|_{M_{t_j}}= id_{M_{t_j}}$  and  $\varphi_j({N_i}_{Q_j})=N_{i,t_j}\times Q_j,$ for $i=1,...,r.$
\end{teo}

Let us recall that a semialgebraic subset of $R^n$ is said of {\it complexity} at most $(p,q)$ if it admits a description as follows
 $$ \bigcup_{i=1}^s\bigcap_{j=1}^{k_i}\{x\in R^n | f_{ij}(x) \ast_{ij}0 \},$$ where $f_{ij}\in R[X_1,...,X_n]$,  and
$\ast_{ij}\in \{<, >, =\}$, $\Sigma _{i=1}^s k_i\leq p$,
$\mbox{deg}(f_{ij}) \leq q$ for $i=1,...,s$ and $j=1,...,r_i$. We say that a semialgebraic map has {\it complexity}
at most $(p,q)$ if its graph has complexity at most $(p,q)$.

The following result holds for any real closed field.
\begin{teo}(\cite{D}, Proposition 3.2)\label{prepa}
There exist a semialgebraic subset $\mathcal{A}(n,p,q)$  in some affine space $R^{\alpha(n,p,q)}$ and a semialgebraic family $\mathcal{S}(n,p,q)\subset R^n\times \mathcal{A}(n,p,q)$  such that:\\
 (i) For every $a\in \mathcal{A}(n,p,q)$ the fiber 
  $$\mathcal{S}_a(n,p,q)=\{x\in R^n: (x,a)\in \mathcal{S}(n,p,q)\}$$
  is a semialgebraic subset of complexity at most $(p,q)$ of $R^n$\\
  (ii) For every semialgebraic subset $S\subset R^n$ of complexity at most  $(p,q)$, there is  $a\in \mathcal{A}(n,p,q)$ such that: $S=\mathcal{S}_a(n,p,q).$\\
 $\mathcal{A}(n,p,q)$ and $\mathcal{S}(n,p,q)$ are defined in a uniform way by  first order formulas of the theory of real closed fields without parameters which can be effectively constructed from $n,p,q$.
\end{teo}

Let us observe that there is another equivalent definition of complexity of semialgebraic sets in \cite{CRS1} Section 2 which induces the sames properties of semialgebraic family of semialgebraic subset of bounded complexity. 

We call {\it Nash NC-system} any family $(M;N_1,...,N_r)$ such that $M$ is an affine Nash manifold and $N_i$ for $i=1,...,r$ are closed Nash submanifolds of $M$ having simultaneously only normal crossings. The {\it complexity} of a Nash NC-system $(M;N_1,...,N_r)$ is bounded by  $(p,q)$ if $M$ and $N_1,...,N_r$ have complexities bounded by $(p,q).$

As a consequence of Nash isotopy theorem, we prove that there is finitely and effectively many models of Nash NC-systems with bounded complexity. 
\begin{teo}\label{sdf}
Given the integers $n, p, q \,\, \mbox{and}\,\, r$, there  exist a triplet of integers $(u,v,w)$ such that for each $i=1,...,u$ there is a Nash NC-system $(M_i;N_{i_1},...,N_{i_r})$ of complexity bounded by $(p,q)$ and for any Nash NC-system $(M;N_1,...,N_r)$ of complexity bounded by $(p,q)$ there are  $i\in\{1,...,u\} $ and a Nash diffeomorphism $\Phi:M\rightarrow M_i$ of complexity bounded by $(v,w)$ and such that $\Phi(N_j)=N_{i_j}$ for $j=1,...,r.$

Moreover, $u,v,w$ are bounded by recursive functions of $n, p, q\,\, \mbox{and}\,\, r.$
\end{teo}
\begin{proof} 
By Theorem \ref{prepa},  the parameters space $\mathcal{A}(n,p,q)$ of semialgebraic subsets of $R^n$ of complexity bounded by $(p,q)$  is a semialgebraic set.
then the parameters space of Nash submanifolds of  $R^n$ of complexity bounded by  $(p,q)$ is included in $\mathcal{A}(n,p,q)$  and is a semialgebraic subset
of $\mathcal{A}(n,p,q)$. Indeed, the condition for a semialgebraic subset to be a Nash manifold  can be translated  into a first order formula of the theory of real closed fields (see \cite{R}, Proposition 3.5 and \cite{D}, Proposition 5.7 (i)). Let us denote this set by $\mathcal{N}(n,p,q)$.

Let $\mathcal{NF}(n,p,q)\subset R^n\times \mathcal{N}(n,p,q)$ be a semialgebraic family of  Nash manifolds of complexity at most $(p,q).$ It is of course a semialgebraic subset of $\mathcal{S}(n,p,q)$ for the same reasons as above.  Consider the set $\mathcal{NC}(n,p,q)\subset \mathcal{N}(n,p,q)\times \mathcal{N}(n,p,q)\times ...\times \mathcal{N}(n,p,q)\,\, (r+1 \mbox{copies})$ consisting of the $(t, t_1,...,t_r)$ such that $$(\mathcal{NF}(n,p,q)_t;\mathcal{NF}(n,p,q)_{t_1},...,\mathcal{NF}(n,p,q)_{t_r})$$ is a Nash NC-system with complexity at most $(p,q).$ By the proof of Lemme \ref{croisement normaux} the normal crossing property also can be translated into a first order formula of the theory of real closed fields. It follows that $\mathcal{NC}(n,p,q)$  is a semialgebraic set. Then we have the semialgebraic families $\mathcal{M}=\{(x,s)\in R^n\times\mathcal{NC}(n,p,q): x\in \mathcal{NF}(n,p,q)_t\},$ $\mathcal{M}_1=\{(x,s)\in R^n\times\mathcal{NC}(n,p,q): x\in \mathcal{NF}(n,p,q)_{t_1}\},$..., $\mathcal{M}_r=\{(x,s)\in R^n\times\mathcal{NC}(n,p,q): x\in \mathcal{NF}(n,p,q)_{t_r}\}.$ Then use fibers at $\alpha\in \widetilde{\mathcal{NC}(n,p,q)}$ to find a Nash submanifold $Q\subset \mathcal{NC}(n,p,q)$ such that $\alpha\in \widetilde{Q}$ and the restrictions to $Q$ of the families give a Nash NC-system. Applying Theorem \ref{isotopy} we make it trivial. By the real spectrum compactness, there are finitely many $Q's$ cover $\mathcal{NC}(n,p,q).$

Using the properties of the finiteness of Nash trivialization of the family of Nash NC-systems in $R^n$ of complexity bounded by $(p,q)$ and \cite{D}, Proposition 5.9, we get the existence of the couple of  integers $(v,w)$ such that for any Nash NC-systems $(M;N_1,...,N_r)$ and $(M';N_1',...,N_r')$ in $R^n$ of complexity bounded by $(p,q)$ which are Nash diffeomorphic, there exists a Nash diffeomorphism $\Phi:M\rightarrow M'$  whose complexity is bounded by $(v,w)$ and such that $\Phi(N_i)=N_i',$ for $i=1,...,r.$

%Using the technique in \cite{D} we can parametrize the Nash NC-systems of complexity bounded by $(p,q)$ by a semialgebraic subset defined by a first order formula of the theory of real closed fields with coefficients in $\mathbb{Z}.$ Indeed we know that to be a Nash manifold of complexty bounded by $(p,q)$ can be translated into a first order formula of the theory of real closed fields (see \cite{R}, Proposition 3.5 and \cite{D}, Proposition 5.7 (i)) and by the proof of Lemme \ref{croisement normaux} the normal crossing property also can be translated into a first order formula of the theory of real closed fields.

It follows that there is a triplet of integers $(u,v,w)$ such that for each $i=1,...,u$ there is a Nash NC-system $(M_i;N_{i_1},...,N_{i_r})$ of complexity bounded by $(p,q)$ and for any Nash NC-system $(M;N_1,...,N_r)$ of complexity bounded by $(p,q)$ there is $i\in\{1,...,u\} $ and a Nash diffeomorphism $\Phi:M\rightarrow M_i$  of complexity bounded by $(v,w)$ and such that $\Phi(N_j)=N_{i_j}$ for $j=1,...,r.$ 

So fixing $n, p, q\,\,\mbox{and}\,\, r,$  there exists $(u,v,w)\in \N^3$ such that the following formula holds:

 $\Psi (n, p, q, r, u, v, w):=$
 \begin{quote}
 \virg{{\it There are $u$ Nash NC-systems $(M_i;N_{i_1},...,N_{i_r})$, $i=1,...,u,$  with complexity bounded by $(p,q)$ such that for every Nash NC-system 
$(M; N_1,...,N_r)$ in  $\R^n$ of complexity at most $(p, q)$ there is
 $i\in\{1,...,u\} $ and a Nash diffeomorphism $\Phi:M\rightarrow M_i$  of complexity bounded by $(v,w)$ and such that $\Phi(N_j)=N_{i_j}$ for $j=1,...,r.$}}
\end{quote}
Of course the formula $\Psi (n, p, q, r, u, v, w)$ can be translated into a sentence of the first order theory of real closed fields. So
there is an algorithm with input $n, p, q, r, u, v, w$  and output the sentence $\Psi (n, p, q, r, u, v, w).$ Here is an algorithm to produce $u, v$ and $w$ from $(n,p,q,r)$:
\begin{itemize}
\item Set $u:=1, v:=1, w:=1$.
\item While $\Psi (n, p, q, r, u, v, w)$ is false, do $u:=u+1, v:=v+1, w:=w+1.$
\end{itemize}
So the decidability of the  first order theory of real-closed  fields, due to Tarski,  implies that the function $$(n,p,q,r)\mapsto (u(n,p,q,r), v(n,p,q,r), w(n,p,q,r)),$$  is bounded by a recursive function. This achieves the proof.
\end{proof}

\section{Desingularization on the Nash category over an arbitrary real closed field}

In [BM1] and [BM2], the authors prove that the desingularization theorem holds true for the objects $X$ belonging to the category $\mathcal{A}$  of locally ringed spaces over a field $k$ of characteristic zero which satisfies the four properties that we shall directly verify for the category $\textbf{Nash}_R$ which interest us in this paper.
\begin{enumerate}[(i)]
\item
Every locally ringed space $(\widetilde{X}, \mathcal{N}_{\widetilde{X}})$ in the category  $\textbf{Nash}_R$ is a closed subspace $(Z(\mathcal{I}), \mathcal{N}_{\widetilde{M}}/\mathcal{I})$ of a Nash manifold  $(\widetilde{M},\mathcal{N}_{\widetilde{M}})$ which is  a smooth space (all its local rings are regular local rings). 
\item
For $(\widetilde{M},\mathcal{N}_{\widetilde{M}})$ a smooth space in the category $\textbf{Nash}_R,$ there is a  neighbourhood basis  given by (the support of) Nash coordinate charts $\widetilde{U}$. That is, there is  coordinates system  $(x_1,...,x_m)$ on $\widetilde{U}$consisting of  Nash   functions on $\widetilde{U}.$ Indeed each $x_i \in \mathcal{N}_{\widetilde{M}}(\widetilde{U})$ and the partial derivatives 
$\frac{\partial^{|\beta|}}{\partial x^{\beta}}= \frac{\partial^{\beta_1+...+\beta_m}}{\partial x_1^{\beta_1}...x_m^{\beta_m}}$ make sense  as transformations $\mathcal{N}_{\widetilde{M}}(\widetilde{U})\rightarrow \mathcal{N}_{\widetilde{M}}(\widetilde{U})).$\\ 
Furthermore, for each $\alpha\in \widetilde{U},$ one can verify that the Taylor homomorphism $T_{\alpha}$ at 
${\alpha}$ commutes with differentiation: $T_{\alpha} \circ (\frac{\partial^{|\beta|}}{\partial x^{\beta}})=\frac{\partial^{|\beta|}}{\partial x^{\beta}}\circ T_{\alpha}$, for all $\beta \in \mathbb{N}^r.$ By Section 2, $T_{\alpha}$ induces an isomorphism   $\widehat{\mathcal{N}}_{\alpha}\cong k(\alpha)[[X_1,...,X_r]]$ with $r=\mbox{dim}(M)-\mbox{dim}( \alpha).$
\item
For each $\widetilde{X}\in \textbf{Nash}_R$, $\mathcal{N}_{\widetilde{X}}$ is a coherent sheaf  of ideals of Nash functions and $\widetilde{X},$  in this category, as stated in Section 2,  is noetherian. Every decreasing sequence of closed subspaces of $\widetilde{X}$ stabilizes since $\mathcal{N}_{\widetilde{X}}(\widetilde{X})$ is a noetherian ring.
\end{enumerate}

It remains to show that $\textbf{Nash}_R$ is closed under blowings-up. It is clear that the blowing-up of a Nash manifold gives a Nash manifold. But in our case we consider only  {\it affine}  Nash manifolds, and it is not clear that the Nash manifold obtained by blowing-up (which is locally defined) is affine. For instance consider $\mathbb{R}/\mathbb{Z}$ and the Nash open sets $(0,1)$ and $(\tfrac{1}{2}, \tfrac{3}{2}).$ We define a Nash structure on $\mathbb{R}/\mathbb{Z}$ by glueing the charts $(0,1)$ and $(\tfrac{1}{2}, \tfrac{3}{2})$ in the following way: 
\begin{eqnarray*}
(0,1)\supset (\tfrac{1}{2}, 1)\stackrel{Id}{\longrightarrow} (\tfrac{1}{2}, 1)\subset (\tfrac{1}{2}, \tfrac{3}{2})\, \, \mbox{and}
\end{eqnarray*}
\begin{eqnarray*}
(0,1)\supset (0,\tfrac{1}{2})&\stackrel{\sim}{\longrightarrow}& (1,\tfrac{3}{2})\subset (\tfrac{1}{2}, \tfrac{3}{2})\\
t&\mapsto&t+1,
\end{eqnarray*}
 Consider $\varphi\in \mathcal{N}_{\mathbb{R}/\mathbb{Z}}.$ The canonical mapping $f: \mathbb{R}\rightarrow \mathbb{R}/\mathbb{Z}$ is a Nash map. Then   
 $\varphi\circ f$ is constant because it is periodic and a periodic Nash map is constant. That implies that $\varphi$ is constant. Then $\mathbb{R}/\mathbb{Z}$ is non-affine.

 In the following subsection we prove that the blowing-up of an affine Nash  manifold gives an affine Nash manifold. 

 \vspace{0.35cm}
 
 \textbf{Global Definition of Blowing-up in Nash category}
 \newline   
 For the sake of simplicity, although we work on the category $\textbf{Nash}_R,$  we  shall write here the objects without the tilde symbol  $\, \widetilde{}$ . So, let $C\subset M$ be a closed Nash submanifold of an affine Nash manifold $M$ of codimension $c$. Let $\mathcal{I}_{C}$ be the  sheaf  of ideals  of germs of Nash functions vanishing  on $C$. It is known that the sheaf $\mathcal{I}_C$ is a finite sheaf (cf.\cite{S}). It follows from \cite{CRS}  that there  are finitely many $f_1,...,f_r\in\mathcal{N}_{M}(M)$ such that $\mathcal{I}_{C}= (f_1,...,f_r)\mathcal{N}_M.$
Consider the  semialgebraic subset $M' \subset M\times \mathbb{P}^{r-1}$ defined as follows:
\begin{equation*} 
M'= \mbox{clos}(\{(x,[\xi_1,...,\xi_r])\in (M\setminus C)\times \mathbb{P}^{r-1}: f_i(x)\xi_j= f_j(x)\xi_i; i,j=1,...,r\}),
\end{equation*}
where clos defines the closure in $M\times \mathbb{P}^{r-1}$ with respect to the  usual topology.  Let us show that $M'$ is a blowing-up of $M$ with center $C.$ Consider the projection  $\Pi: M\times \mathbb{P}^{r-1}\rightarrow M,$ then the restriction 
$\Pi|_{M'}: M'\rightarrow M.$ It follows from the construction of $M'$ that $\Pi|_{\Pi^{-1}(M\setminus C)}: \Pi^{-1}(M\setminus C)\rightarrow M\setminus C$ is a Nash diffeomorphism.

Now for any $\alpha_0\in C$, there is a semialgebraic  open neighbourhood $U$ of $\alpha_0$ in $M$ such that $c$ of the $f_i$ belong to a coordinates system over $U$, say $f_1,...,f_c, x_{c+1},...,x_m$ for simplicity, such that for $c<j\leq r$ there are $h_{j,k}\in \mathcal{N}_{M}(U)$ with $f_j= h_{j,1}f_1+...+h_{j,c}f_c$ and $k=1,...,c.$  By the usual blowing-up of $U$ with center $U\cap C=\{f_1=...=f_c=0\}$, we get: $$U'=\{(x, [\xi_1,...,\xi_c])\in U\times \mathbb{P}^{c-1}: f_i(x)\xi_j= f_j(x)\xi_i,   1\leq i,j\leq c\}.$$
It is clear that $U'$ is a Nash  submanifold of  $U\times \mathbb{P}^{c-1}$. Let  us show   that $U'$ is  Nash diffeomorphic to $M'\cap \Pi^{-1}(U).$  Let us  consider the embedding  $\Phi: U\times \mathbb{P}^{c-1}\rightarrow U\times \mathbb{P}^{r-1}$ defined by:
$$(x,[\xi_1,...,\xi_c])\mapsto (x, [\xi_1,...,\xi_c,\sum_{i=1}^c h_{{c+1},i}(x)\xi_i,...,\sum_{i=1}^c h_{r,i}(x)\xi_i]).$$
We \it{claim} \rm that:   $\Phi(U')= M'\cap \Pi^{-1}(U).$ 
It follows that $M'$ is the blowing-up of $M$ with center $C$ and $M'$ is an affine Nash manifold.
\medskip

\noindent {\it Proof of claim:}  
It is clear that  $\Phi(U'\setminus (C\times \mathbb{P}^{c-1}))=M'\cap ((U\setminus C)\times \mathbb{P}^{r-1}).$ Since $(U'\setminus C)\times \mathbb{P}^{c-1}$ is dense in $U'$
 and $\Phi$ is a homeomorphism onto its image, we deduce that $M'\cap ((U\setminus C)\times \mathbb{P}^{r-1})$ is dense in $\Phi(U').$ Hence $M'\cap (U\times  \mathbb{P}^{r-1})\supset M'\cap ((U\setminus C)\times \mathbb{P}^{r-1}),$ is dense in $\Phi(U')$ too. But
$M'\cap(U\times \mathbb{P}^{r-1}) $ is a closed submanifold of $U\times  \mathbb{P}^{r-1},$ hence closed in $\Phi(U')\subset U\times  \mathbb{P}^{r-1}.$
Consequently, $M'\cap (U\times  \mathbb{P}^{r-1})=\Phi(U')$ This completes the proof.
\medskip 
	
Which verifies that $\textbf{Nash}_R$ is closed under blowings-up.We conclude that the desingularization theorem is valid in the category 
$\textbf{Nash}_R.$

\section{Nash trivial simultaneous resolution}

In this section, we prove the global case of the result of Koike ([K1], Theorem I, p.317) which is stated in  this paper by \textbf{Theorem (I)}. In other  words, we prove the global case of the finiteness theorem of Nash trivial simultaneous resolution. 

Let us first  prove  the following technical lemma:
\begin{lem}\label{Centre sur la fibre}
Let $\alpha \in \widetilde{R^s}$, and let $N_{\alpha}$ be a Nash submanifold of  $k(\alpha)^m.$ Let 
$\sigma_{\alpha}:N_{1\alpha}\rightarrow N_{\alpha}$ be the  blowing-up of $N_{\alpha}$ with center $C_{\alpha}$ a Nash submanifold. Then there exist a Nash submanifold $Q\subset R^s$, with $\alpha \in \widetilde{Q}$, Nash manifolds $N_{1Q}$, $N_Q$ and $C_Q\subset N_Q\subset R^m\times Q$ and  a Nash map $\sigma_{Q}:N_{1Q}\rightarrow N_ Q$  the Blowing-up of $N_ Q$ with center $C_Q$ such that and the projections $N_Q\rightarrow Q$ and $C_Q\rightarrow Q$ are Nash onto submersions.
\end{lem}
\begin{proof}
Since $\sigma_{\alpha}$ is a Nash map, its graph $\Gamma_{\sigma_{\alpha}}$ is a semialgebraic set. Then we may write it as follows
$$\Gamma_{\sigma_{\alpha}}=\{(x,y)\in N_{1\alpha}\times N_{\alpha}: \Phi(x,y)\}$$ where $\Phi(x,y)$ is a first order formula in the theory of real closed fields with coefficients $a=(a_1,...,a_n)\in k(\alpha)^n$ for some $n\in \mathbb{N}.$ We may also write $C_{\alpha}=\{x\in N_{\alpha}: \Psi(x)\}$ where $\Psi(x)$ is a first order formula in the theory of real closed fields with coefficients $b=(b_1,...,b_c)\in k(\alpha)^c$ for some $c\in \mathbb{N}.$

The fact that the Nash map $\sigma_{\alpha}:N_{1\alpha}\rightarrow N_{\alpha}$ is the blowing-up of 
$ N_{\alpha}$ with center $C_{\alpha}$ can be expressed into a first order formula of theory of  real closed fields (see the construction of global definition of Blowings-up in Nash category in Section 2). It follows that  there exist a Nash submanifold $Q\subset R^s$ with $\alpha\in \widetilde{Q},$ a Nash submanifold $C_Q\subset N_Q$  and $\sigma_Q: N_{1Q}\rightarrow N_Q$  is the blowing-up of $N_Q$ with smooth center $C_Q.$ Applying also Theorem \ref{submersion}, shrinking $Q$ if necessary, the projection $N_{1Q}\rightarrow Q,$  $N_{Q}\rightarrow Q,$ and $C_Q\rightarrow Q$ are Nash submersions. This ends the proof.
\end{proof}

Now we will prove the \textbf{Theorem (I)}.
Let $N$ be an affine  Nash manifold  and $J$ a semialgebraic  subset of $\mathbb{R}^s$. Let $F: N\times J\rightarrow \mathbb{R}^k$ be a semialgebraic family of Nash maps $f_t: N \rightarrow \mathbb{R}^k$ ($t\in J$). 

\begin{proof}[Proof of \textbf{Theorem (I)}]
As seen in Section 3, Nash desingularization  is valid over any real closed field. Let $\alpha \in \tilde{J}.$ By the Theorem \ref{family1}, $f_{\alpha}$ is a Nash map. It follows that  the fiber $f_{\alpha}^{-1}(0)\subset N_{\alpha}$ is  a Nash set. By the Nash Desingularization Theorem, there exists a Nash resolution of $f_{\alpha}^{-1}(0)$ in $N_{\alpha}$, $\Pi_{\alpha}:  \mathcal{M}_{\alpha} \rightarrow  N_{\alpha}$. The properties that $\Pi_{\alpha}$ must  satisfy to be a Nash resolution  can be translated in a first order formula of the theory of real closed fields (a consequence of the construction of global definition of blowing-up in Nash category in Section 3).   Consider the Nash divisors $\mathcal{D}_{1\alpha},...,\mathcal{D}_{k\alpha}$  and  $\mathcal{V}_{\alpha}$ the strict transform of $V_{\alpha}=f_{\alpha}^{-1}(0)$ with respect to $\Pi_{\alpha}$. By the definition of $\Pi_{\alpha}$, we have that,  $\mathcal{D}_{1\alpha},...,\mathcal{D}_{k\alpha}$ and $\mathcal{V}_{\alpha}$ have only normal crossings.

So by the Theorem \ref{family}, Theorem \ref{submersion}, Lemma \ref{Centre sur la fibre} and Lemma \ref{croisement normaux} these data are fibers of suitable families, and there exists a Nash submanifold $Q\subset J$, which can be supposed Nash diffeomorphic to an open simplex  in an euclidean space, and a Nash map $F_Q:N\times Q \rightarrow \mathbb{R}^k,$ such that $\alpha\in \widetilde{Q}$ and we have a commutative  diagram of Nash mapping 

$$\begin{CD}
\mathcal{M}_{Q} @>>> Q\\
@V{\Pi}VV @|\\
N\times Q @>p>> Q
\end{CD},$$ 
where  $p\circ\Pi $ is a  Nash submersion and the map $\Pi$ is a Nash resolution of 
$V= F_{Q}^{-1}(0)$ in $N\times Q$ and the strict transform $\mathcal{V}$  of $V$, and the exceptional divisors $\mathcal{D}_1,...,\mathcal{D}_k$ have only normal crossings. Furthermore  
\begin{itemize}
   \item $ p\circ \Pi|_{\mathcal{V}}: \mathcal{V}\rightarrow Q,$ 
   \item  $p\circ \Pi|_{\mathcal{D}_{j_1}\cap...\cap \mathcal{D}_{j_s}}: \mathcal{D}_{j_1}\cap...\cap \mathcal{D}_{j_s} \rightarrow Q,$
    \item $p\circ \Pi|_{\mathcal{V}'\cap \mathcal{D}_{j_1}\cap...\cap \mathcal{D}_{j_s}}:\mathcal{V}\cap \mathcal{D}_{j_1}\cap...\cap  \mathcal{D}_{j_s}\rightarrow Q $   $(1\leq j_1<...<j_s\leq k)$
  \end{itemize}
are Nash submersions.

Applying Theorem \ref{isotopy} to $p\circ \Pi$, there is a finite partition of $Q$ into Nash open simplices $Q_j,$ $j=1,...,b,$ and for any $j,$ there are a point $t_j\in Q_j$ and  a Nash diffeomorphism $\Phi_j: \mathcal{M}_{Q_j}\rightarrow \mathcal{M}_{t_j}\times Q_j$ such that $\Phi_j(\mathcal{V}_j)=\mathcal{V}_{t_j}\times Q_j$, $\Phi_j(\mathcal{D}_i)=\mathcal{D}_{it_j}\times Q_j$, $i=1,..., r$  and the following diagram is commutative:
$$\begin{CD}
\mathcal{M}_{Q_i} @>{\Phi_j}>>  \mathcal{M}_{t_j}\times Q_j\\
@V{\Pi_j}VV @VVpV\\
N\times Q_j @>p>> Q_j
\end{CD},$$ where  $\mathcal{M}_{t_j}= (p\circ \Pi)^{-1}(t_j),$ $\mathcal{V}_{t_j}= \mathcal{V}_j\cap (p\circ \Pi)^{-1}(t_j)$ and $\mathcal{D}_{it_j}=\mathcal{D}_i\cap (p\circ \Pi)^{-1}(t_j)$ with $t_j\in Q_j.$ Now, we choose $j_0\in\{1,...,b\}$ such that $\alpha\in \widetilde{Q_{j_0}}$. Since $\Pi_{\alpha}$ is a Nash resolution of $V_{\alpha}$, shrinking $Q_{j_0}$ as necessarily to preserve the resolution property, it follows that $\Pi_{j_0}$ gives a Nash trivial simultaneous resolution of $V_j$ in $N\times Q_j$ over $Q_j.$  Let us write it $\widetilde{Q^{\alpha}}.$

Since $\widetilde{J}= \bigcup_{\alpha \in \widetilde{J}} \widetilde{Q^{\alpha}}$ with $Q^{\alpha}\subset J$ a Nash submanifold diffeomorphic to an open simplex in an Euclidean space. Over  each $\widetilde{Q^{\alpha}}$ there is a Nash trivial simultaneous resolution of $F_{Q^{\alpha}}^{-1}(0)$ in $N\times Q^{\alpha}$. Since $\widetilde{J}$ is compact there is a finite subcovering $\widetilde{J}= \bigcup_{i=1}^{p} \widetilde{Q^{\alpha_i}}$. Set $J_i=Q^{\alpha_i}$, and this ends the proof.
\end{proof}

Bounding complexities of Nash subsets and applying \textbf{Theorem (I)}, we get that there are finitely and effectively many models of Nash subsets with theirs  Nash resolutions.

\begin{teo}
Given the integers $n, p\,\,\mbox{and}\,\ q $, there  exists a triplet of integers $(u,v,v)$ such that for each $i=1,...,u$ there are a Nash subset $V_i\subset \mathbb{R}^n$ of complexity bounded by $(p,q)$ and  a Nash resolution $\Pi_i:\mathcal{M}_i\rightarrow \mathbb{R}^n$ of $V_i$; and for any Nash subset $V\subset \mathbb{R}^n$ of complexity bounded by $(p,q)$ there are a Nash resolution $\Pi:\mathcal{M}\rightarrow \mathbb{R}^n$ of $V$, $i\in\{1,...,u\} $ and a Nash diffeomorphism $\Phi:\mathcal{M}\rightarrow \mathcal{M}_i,$  of complexity bounded by $(v,w),$ which sends  strict transform  to strict transform  and exceptional  divisors to exceptional divisors.

Moreover,  $u,v\,\,\mbox{and}\,\,w$ are recursive functions of $n, p\,\, \mbox{and}\,\, q.$
\end{teo}
\begin{proof}

Let us note that the family of Nash subsets of $\mathbb{R}^n$ of complexities bounded by $(p,q),$ say $\mathcal{EN}(n,p,q),$ can be parametrized by a semialgebraic set denoted by
$\mathcal{E}(n,p,q).$ It is also straightforward to show that $\mathcal{EN}(n,p,q)$ is semialgebraic. 

Now, set $F: \mathbb{R}^n\times \mathcal{E}(n,p,q)\rightarrow \mathbb{R}^k$ a semialgebraic family of Nash maps, for a suitable $k,$ such that $F_{\mathcal{E}(n,p,q)}^{-1}(0)=\mathcal{EN}(n,p,q).$  Applying \textbf{Theorem (I)} to $F,$ there exists a finite partition $\mathcal{E}(n,p,q)=\mathcal{E}_1\cup ...\cup \mathcal{E}_u$ which satisfies the following conditions:
\begin{enumerate}
\item    Each $\mathcal{E}_i$ is Nash manifold which is Nash diffeomorphic to an open  simplex in some Euclidean space, and $F_{\mathcal{E}_i}$ is  Nash.
\item  For each $i$, $F_{\mathcal{E}_i}^{-1}(0)$ admits a Nash trivial simultaneous resolution along $\mathcal{E}_i.$
 \end{enumerate}
For each $i\in\{1,...,u\},$ take a point $t_i\in \mathcal{E}_i$ and set $V_i=F_{t_i}^{-1}(0)$ and $\Pi_i:\mathcal{M}_{t_i}\rightarrow \mathbb{R}^n$ a Nash resolution of $V_i$ which is a fiber at $t_i$ of the Nash simultaneous resolution of the family $F_{\mathcal{E}_i}^{-1}(0).$ So for any Nash subset $V\subset R^n$ of complexity bounded by $(p,q)$, by (2) there are $i\in \{1,...,u\}$ and $t\in \mathcal{E}_i$ such that $V= F_{t}^{-1}(0)$ and a Nash resolution  $\Pi_{t}:\mathcal{M}_{t}\rightarrow \mathbb{R}^n$ of $V$ (a fiber at $t$ of the Nash simultaneous resolution of the family $F_{\mathcal{E}_i}^{-1}(0)).$ Then by (2), one can see that $\mathcal{M}_{t}$ is Nash diffeomorphic to $\mathcal{M}_{t_i}.$ Now applying \cite{D}, Proposition 5.9, we get the existence of the couple of  integers $(v,w)$  and a needed Nash diffeomorphism $\Phi_i$ whose complexity is bounded by $(v,w)$. 

For the effectiveness  of the function $(n,p,q)\rightarrow (u(n,p,q),v(n,p,q), w(n,p,q))$, use the technique in the proof of Theorem \ref{sdf}.
Which ends the proof.
\end{proof}

\section{Finiteness Theorem on Blow-Nash Triviality}

Now, we can prove our second result. Let $N$ be an affine Nash manifold  and $J$ a semialgebraic  subset of $\mathbb{R}^s$. Let $F: N\times J\rightarrow \mathbb{R}^k$ be semialgebraic family of Nash maps $f_t: N \rightarrow \mathbb{R}^k$     such that   
for every $t\in J$ the Nash set $f_t^{-1}(0)$  has only isolated singularities.  

\begin{proof}[Proof of \textbf{Theorem (II)}]
Note that the singular set $S(f_t^{-1}(0))$ of the Nash set $f_t^{-1}(0)$ is discret and semialgebraic, then it is finite and moreover its cardinal is uniformly bounded (as a consequence of uniform finiteness in semialgebraic families). This implies that, given $\alpha\in \widetilde{J},$ the Nash set $f_{\alpha}^{-1}(0)$ has finitely many isolated singular points, say $s_{1,\alpha},...,s_{\nu,\alpha}.$ Of course the singular set of $f_{\alpha}^{-1}(0)$ is defined  as follows $S(f_{\alpha}^{-1}(0))=\{x\in N_{\alpha}:F(x,t(\alpha))=0,\mbox{rank}(\frac{\partial}{\partial x}F(x,t(\alpha)))<k\}.$ By the Nash desingularization Theorem, there exists a Nash resolution of $f_{\alpha}^{-1}(0)$ in $N_{\alpha}$, $\Pi_{\alpha}:  \mathcal{M}_{\alpha} \rightarrow  N_{\alpha}$.  We index the Nash exceptional divisors as 
$\mathcal{D}_{i,j\alpha}$ for $i=1,...,\nu$ and $j=1,...,\mu(i)$ with the property $\Pi_{\alpha}(\mathcal{D}_{i,j\alpha})=\{s_{i,\alpha}\}$
  and  $\mathcal{V}_{\alpha}$ the strict transform of $V_{\alpha}=f_{\alpha}^{-1}(0)$ with respect to $\Pi_{\alpha}$. By the definition of $\Pi_{\alpha}$, we have that,  $\mathcal{D}_{i,j\alpha}$ for $i=1,...,\nu,\,\,j=1,...,\mu(i)$ and $\mathcal{V}_{\alpha}$ have only normal crossings.
	
So by the Theorem \ref{family}, Theorem \ref{submersion}, Lemma \ref{Centre sur la fibre} and Lemma \ref{croisement normaux} there exists a Nash submanifold $Q\subset J$, which can be supposed Nash diffeomorphic to an open simplex  in an euclidean space, such that $\alpha\in \widetilde{Q}$ and the following diagram  commutes:

$$\begin{CD}
\mathcal{M}_{Q} @>>> Q\\
@V{\Pi}VV @|\\
N\times Q @>p>> Q
\end{CD},$$ 
where  $p\circ\Pi $ is a  Nash submersion and the map $\Pi$ is a Nash resolution of 
$V= F_{Q}^{-1}(0)$ in $N\times Q$  such that  $\mathcal{V}$ (the strict transform of $V$) and  $\mathcal{D}_{i,j}$ the Nash divisors for 
$i=1,...,\nu,\,\,j=1,...,\mu(i)$ simultaneously have only normal crossings 
and  
\begin{itemize}
   \item $ p\circ \Pi|_{\mathcal{V}}: \mathcal{V}\rightarrow Q,$ 
   \item  $p\circ \Pi|_{\mathcal{D}_{i,j_1}\cap...\cap \mathcal{D}_{i,j_s}}: \mathcal{D}_{i,j_1}\cap...\cap \mathcal{D}_{i,j_s} \rightarrow Q,$
    \item $p\circ \Pi|_{\mathcal{V}\cap \mathcal{D}_{i,j_1}\cap...\cap \mathcal{D}_{i,j_s}}:\mathcal{V}\cap \mathcal{D}_{i,j_1}\cap...\cap   \mathcal{D}_{i,j_s}\rightarrow Q $ $(i=1,...,\nu$  and  $1\leq j_1<...<j_s\leq \mu(i))$
  \end{itemize}
are Nash submersions. There is also a semialgebraic set $$S(f_Q^{-1}(0))=\{(x,t)\in N\times Q;F(x,t)=0,\mbox{rank}(\frac{\partial}{\partial x}F_Q)<k,\frac{\partial}{\partial t}F_Q=0 \}$$ whose fiber at $\alpha$ is $S(f_{\alpha}^{-1}(0))$ and continuous semialgebraic maps $s_i:Q\rightarrow N$ for $i=1,...,\nu$ such that $S(f_Q^{-1}(0))=\{(s_i(t),t):i=1,...,\nu \,\, \mbox{and}\,\, t\in Q\}$ and  $F(s_i(t),t)\equiv 0$ for $i=1,...,\nu$ over $Q$ (shrinking $Q$ if necessarily to preserve the singular property of its fiber at $\alpha).$ It is clearly a singular set of $V.$ From this, we get  and $\Pi(\mathcal{D}_{i,j})\subset \Gamma_{s_i}$ for $j=1,...,\mu(i)$ where $\Gamma_{s_i}$ is the graph of the map $s_i.$ 

Now, applying Theorem \ref{isotopy} to $p\circ \Pi$, we may shrink $Q$ so that $\alpha\in \widetilde{Q}$ and  there exist $t_0\in Q$ and a Nash diffeomorphism $\Phi: \mathcal{M}_{Q}\rightarrow M_{t_0}\times Q$ such that $\Phi(\mathcal{V})=\mathcal{V}_{t_0}\times Q$, $\Phi(\mathcal{D}_{i,j})=\mathcal{D}_{i,j,t_0}\times Q$, $i=1,...,\nu,\,\,j=1,...,\mu(i)$  and the following diagram is commutative:
$$\begin{CD}
\mathcal{M}_{Q} @>{\Phi}>>  \mathcal{M}_{t_0}\times Q\\
@V{\Pi}VV @VVpV\\
N\times Q @>p>> Q
\end{CD},$$ where  $\mathcal{M}_{t_0}= (p\circ \Pi)^{-1}(t_0),$ $\mathcal{V}_{t_0}= \mathcal{V}\cap (p\circ \Pi)^{-1}(t_0)$ and $\mathcal{D}_{i,j,t_0}=\mathcal{D}_{i,j}\cap (p\circ \Pi)^{-1}(t_0).$ It follows that $\Pi$ gives a Nash trivial simultaneous resolution of $V$ in $N\times Q$ over $Q.$ 

Let us construct a semialgebraic homeomorphism  $\phi: N\times Q \rightarrow N\times Q$ such that the following diagram commutes:
$$\begin{CD}
\mathcal{M}_{Q} @>{\Pi}>> N\times Q@>p>>Q\\
@V{\Phi}VV @V{\phi}VV @|\\
\mathcal{M}_{t_0}\times Q @>{{\Pi}_{|{\mathcal{M}_{t_0}}}\times id_{Q}}>> N\times Q@>p>> Q
\end{CD}$$

 So, let us set the map $\phi$ as follows: $\phi(x,t)=(y,t)$ such that
\begin{itemize}
	\item  ${\phi}(x,t)= ({\Pi}_{|\mathcal{M}_{t_0}}\times id_{Q})\circ \Phi \circ {\Pi}^{-1}(x,t)$ for $x\notin \{s_1(t),...,s_{\nu}(t)\}$
	\item $\phi(s_i(t),t)=(s_i(t_0),t)$ for $i=1,...,\nu.$
\end{itemize}
 
Let us show that $\phi$ as constructed is a semialgebraic homeomorphism. Set $S=\{(s_i(t),t):i=1,...,\nu \,\, \mbox{and}\,\, t\in Q\}$ the singular set  of $V=F_{Q}^{-1}(0).$ Of course, ${\phi}_{|(N\times Q)\setminus S}$ is a Nash diffeomorphism as composition of Nash diffeomorphisms and ${\phi}_{|S}$ is a semialgebraic homeomorphism. So, $\phi$ is a semialgebraic map.  It remains to show that $\phi$ is continuous. Indeed, let be $(s_i(t),t)\in S.$ Consider a semialgebraic path $\gamma: [0,1]\rightarrow N\times Q$ with $\gamma([0,1))\subset (N\times Q)\setminus S$ and $\gamma(1)=(s_i(t),t).$ Consider a semialgebraic lifting $\gamma': [0,1]\rightarrow \mathcal{M}_{Q}$ of $\gamma$ in $\mathcal{M}_{Q}$ by $\Pi$ ($\gamma'$ exits because $\Pi$ is proper). It follows that $\gamma'(1)\in  \bigcup_j^{\mu(i)} {D_{i,j}}_t$ where ${D_{i,j}}_t= D_{i,j}\cap (p\circ\Pi )^{-1}(t).$  For $s\in [0,1)$, by definition of $\phi$ we get $$ \phi(\gamma(s))= ({\Pi}_{|{\mathcal{M}_{t}}}\times id_{Q})\circ \Phi \circ {\Pi}_{|(N\times Q)\setminus S}^{-1}(\gamma(s)).$$ It is also clear that 
$\phi(\gamma(1))= ({\Pi}_{|{\mathcal{M}_{t_0}}}\times id_{Q})\circ \Phi(\gamma'(1))= (s_i(t),t_0)$ since\\  
$\gamma'(1)\in  \bigcup_j^{\mu(i)} {D_{i,j}}_t.$
It follows that $\phi$ is a semialgebraic homeomorphism. Hence $(N\times Q, V)$ admits a $\Pi$-Blow-Nash trivialization  over $Q.$
 Let us write  this $Q$ by $Q^{\alpha}.$ 
Since $\widetilde{J}= \bigcup_{\alpha \in \widetilde{J}} \widetilde{Q^{\alpha}}$ with $Q^{\alpha}\subset J$ a Nash submanifold diffeomorphic to an open simplex in an Euclidean space over  which  $(N\times Q^{\alpha}, V)$ admits a $\Pi$-Blow-Nash triviality, by compactness property of $\widetilde{J}$, we get $\widetilde{J}= \bigcup_{i=1}^{p} \widetilde{Q^{\alpha_i}}$. Set $J_i=Q^{\alpha_i}$, and this completes the proof.

\end{proof}

To end this paper, we prove that there are finitely many  Blow-Nash equivalence classes of Nash sets with only isolated singularities and bounded complexities. We get also that the result is effective.
\begin{teo}
Given the integers $n, p\,\,\mbox{and}\,\ q $, there  exists a triplet of integers $(u,v,w)$ such that for each $i=1,...,u$ there is a Nash subset $V_i$ of $\mathbb{R}^n$ of complexity bounded by $(p,q)$ having only isolated singularities and  a Nash resolution $\Pi_i:\mathcal{M}_i\rightarrow \mathbb{R}^n$ of $V_i$; and for any Nash subset $V$ of $\mathbb{R}^n$ of complexity bounded by $(p,q)$ having only isolated singularities  there are $i\in\{1,...,u\} ,$ a Nash resolution $\Pi:\mathcal{M}\rightarrow \mathbb{R}^n$ of $V,$ a Nash diffeomorphism $\Phi:\mathcal{M}\rightarrow \mathcal{M}_i$ and a semialgebraic homeomorphism $\phi:\mathbb{R}^n\rightarrow \mathbb{R}^n$ with complexities bounded by $(v,w)$ and making the following diagram commutative \hspace{0.12cm}
$\begin{CD}
M @>{\Phi}>> M_i\\
@V{\Pi}VV @V{\Pi_i}VV\\
\mathbb{R}^n@>{{\phi}}>> \mathbb{R}^n
\end{CD}$ and $\Phi(\mathcal{V})=\mathcal{V}_i, \Phi(\mathcal{D}_{j})=\mathcal{D}_{i_j},\,\,j\in \{1,...,\mu(i)\},\,\, \phi(V)=V_i$ 
where $\mathcal{V}$ is the strict transform of $V,$  $\mathcal{D}_j$'s are the exceptional divisors of $V$, $\mathcal{V}_i$ is the strict transform of $V_i$ and $\mathcal{D}_{i_j}$'s are the exceptional divisors of $V_i.$

Moreover,  $u,v\,\,\mbox{and}\,\,w$ are bounded by recursive functions of $n, p\,\, \mbox{and}\,\, q.$
\end{teo}

\begin{proof}
First, we parametrize the semialgebraic family of Nash subsets of $\mathbb{R}^n$ of complexities at most $(p,q)$ having only isolated singularities by a  semialgebraic set. Then using the same technique of the proof of  Theorem \ref{sdf} and \textbf{Theorem (II)}, we get the result.
\end{proof}

\begin{center}
 DEMDAH KARTOUE MADY\\
    D\'epartement de Math\'ematiques\\
   Faculté des Sciences Exactes et Appliqu\'ees\\
    Universit\'e de Ndjamena\\
    BP:1027, 1 Route de Farcha\\
    email:\texttt{kartoue@hotmail.com}
    %\qquad \texttt{kartoue.demdah@univ-rennes1.fr}
\end{center}


\begin{thebibliography}{ABB}

\bibitem[BCR]{BCR}  J. Bochnak, M. Coste, and M-F. Roy.
{\em Real algebraic geometry}, volume 36 of ergebnisse der
Mathematikund ihrer Grenzgebiete(3). Springer-Verlag, Berlin, 1998. 
\bibitem[BM1]{BM1} E. Bierstone and P.D. Milman {\em Uniformization of analytic spaces,} Jour.Amer. Math. Soc. 2 (1989), 801-836.
\bibitem[BM2]{BM2} E. Bierstone and P.D. Milman {\em Canonical desingularization in characteristic zero by blowing-up the maximum strata of a local invariant,} Invent. Math. 128 (1997), 207-302.
\bibitem[CRS]{CRS} M. Coste, J.M Ruiz and M. Shiota {\em Gobal Problem on Nash Functions}, Rev. Mat. Complut. 2004,17;N\`um 1, 83-115.
\bibitem[CRS1]{CRS1} M. Coste, J.M Ruiz and M. Shiota {\em Uniform bounds on complexity and transfer of global properties of Nash functions}, J. Reine. Angew. Math. 536 (2001) 209-235.
\bibitem[CS]{CS} M. Coste and M. Shiota {\em Nash triviality in family
of Nash manifolds}, Invent. Math. 108, 349-368 (1992).
\bibitem [CS1]{CS1}  M. Coste and M. Shiota {\em Thom's first isotopy lemma: a semialgebraic version, with uniform bound,} Real Analytic and Algebraic Geometry (Ed. F. Broglia, M. Galbiati, A. Tognoli), Walter de Gruyter, Berlin 83-101, 1995. 
\bibitem[D]{D} Demdah K.Mady {\em h-cobordism and s-cobordism theorems: Transfer over Semialgebraic and Nash categories, Uniform bound and Effectiveness}, Annales de l'Institut Fourier, t. 61, n.5(2011), p. 1573-1597.
\bibitem[FKS] {FKS} T. Fukui, S. Koike and M. Shiota {\em Modified Nash triviality of family of zero-sets of real polynomial mappings}, Annales de l'Institut Fourier, t. 48, n.5, p. 1395-1440, 1998.
\bibitem[H1]{H1} H. Hironaka {\em Resolution of singularities  of an algebraic variety over a field of characteristic zero,} I, II, Ann. of Math. 79 (1964) 109-302.
\bibitem[H2]{H2} H. Hironaka {\em Idealistic exponents of singularity,} J.J. Sylvester Sympos., Johns Hopkins Univ., Baltimore (1976) 52-125.
\bibitem[K1]{K1} S. Koike {\em Nash Trivial simultaneous resolution for a family of zero-sets of Nash mappings,} Mathematisch Zeitschrift.237,313-336 (2000).
\bibitem[K2]{K2} S. Koike {\em Finiteness Theorem for Blow-semialgebraic Triviality of a family of three-dimensional Algebraic sets,} Proc. London. Math. Soc. 105 (3),506-540 (2012).
\bibitem[LT] {LT} F. Lazzeri, A. Tognoli {\em Alcune proprietà degli spazi algebrici,} Ann. Scuola Norm.Sup. Pisa 24, 597-632 (1970).
\bibitem[P]{P} A. Prestel {\em Model theory for the real algebraic
geometer}, Dottorato di Ricerca in Matematica,Dip. Mat. Univ. Pisa.
Istituti Editoriali Poligrafici Internazionali 1998.
\bibitem[R]{R} R. Ramanakoraisina {\em Complexit\'e  des fonctions de
Nash}, Commun. Algebra 17 (n.6),1395-1406, 1989.
\bibitem[S]{S} M. Shiota  {\em  Nash  Manifolds.}
Lecture Notes in Mathematics 1269 . Springer-Verlag, Berlin , 1980.
\end{thebibliography}
\end{document}